\documentclass[12pt,twoside]{amsart}
\usepackage{a4wide, amsmath,amsbsy,amsfonts,amssymb,
stmaryrd,amsthm,mathrsfs,graphicx, amscd, tikz-cd, cancel}
\usepackage[normalem]{ulem}
\usepackage{soul}
\usetikzlibrary{matrix,arrows,decorations.pathmorphing}
\usepackage[active]{srcltx}
\usepackage[
	hypertexnames=false,
	hyperindex,
	pagebackref,  
	pdftex,
	breaklinks=true,
	bookmarks=false,
	colorlinks,
	linkcolor=blue,
	citecolor=red,
	urlcolor=red,
]{hyperref}

\usepackage[all]{xy}
\usepackage{subfig}
\usepackage{tikz}
\usetikzlibrary{shapes,arrows,shadows}
\usetikzlibrary{decorations.markings}
\usepackage{color}
\usepackage[normalem]{ulem}
\usepackage{hyperref}
\usepackage{wrapfig}
\usetikzlibrary{arrows}
\usepackage{pinlabel}
\long\def\symbolfootnote[#1]#2{\begingroup%
\def\thefootnote{\fnsymbol{footnote}}\footnote[#1]{#2}\endgroup}

\newtheorem{innercustomthm}{Theorem}
\newenvironment{customtheorem}[1]
  {\renewcommand\theinnercustomthm{#1}\innercustomthm}
  {\endinnercustomthm}
  
  \newtheorem{innercustomcor}{Corollary}
\newenvironment{customcorollary}[1]
  {\renewcommand\theinnercustomcor{#1}\innercustomcor}
  {\endinnercustomcor}

\newtheorem{theorem}{Theorem}[section]

\newtheorem{proposition}[theorem]{Proposition}

\newtheorem{lemma}[theorem]{Lemma}

\theoremstyle{definition}

\newtheorem{remark}[theorem]{Remark}

\newtheorem*{namedtheorem}{\theoremname}
\newcommand{\theoremname}{testing}

\DeclareMathSymbol{\Alpha}{\mathalpha}{operators}{"41}
\DeclareMathSymbol{\Beta}{\mathalpha}{operators}{"42}
\DeclareMathSymbol{\Epsilon}{\mathalpha}{operators}{"45}
\DeclareMathSymbol{\Zeta}{\mathalpha}{operators}{"5A}
\DeclareMathSymbol{\Eta}{\mathalpha}{operators}{"48}
\DeclareMathSymbol{\Iota}{\mathalpha}{operators}{"49}
\DeclareMathSymbol{\Kappa}{\mathalpha}{operators}{"4B}
\DeclareMathSymbol{\Mu}{\mathalpha}{operators}{"4D}
\DeclareMathSymbol{\Nu}{\mathalpha}{operators}{"4E}
\DeclareMathSymbol{\Omicron}{\mathalpha}{operators}{"4F}
\DeclareMathSymbol{\Rho}{\mathalpha}{operators}{"50}
\DeclareMathSymbol{\Tau}{\mathalpha}{operators}{"54}
\DeclareMathSymbol{\Chi}{\mathalpha}{operators}{"58}
\DeclareMathSymbol{\omicron}{\mathord}{letters}{"6F}

\newcommand{\C}{\mathbb{C}}

\newcommand{\Z}{\mathbb{Z}}
\newcommand{\ZZ}{{\widehat{\mathbb Z}}}

\def\Aut{\operatorname{Aut}}

\def\PGL{\operatorname{PGL}}

\def\Out{\operatorname{Out}}
\def\Mod{\operatorname{Mod}}

\def\Hom{\operatorname{Hom}}

\def\Fix{\operatorname{Fix}}
\def\P{{\mathbb P}}

\def\tS{\widetilde{S}}

\def\dd{\partial}

\def\G{{\Gamma}}

\def\kG{{\check\Gamma}}

\def\hG{{\widehat\Gamma}}

\def\hB{\widehat{B}}

\def\ssm{\smallsetminus}
\def\ol{\overline}
\def\wh{\widehat}

\def\da{\downarrow}

\def\hookra{\hookrightarrow}
\def\co{\colon\thinspace}

\begin{document}

\title[A congruence subgroup property for symmetric mapping class groups]{A congruence subgroup property for symmetric mapping class groups}
 
\author{Marco Boggi}
\address{UFF - Instituto de Matem\'atica e Estat\'{\i}stica - Niter\'oi - RJ 24210-200; Brazil}
\email{marco.boggi@gmail.com}

\begin{abstract}    
We prove the congruence subgroup property for the centralizer of a finite subgroup $G$ in the mapping class group of a hyperbolic oriented 
and connected surface of finite topological type $S$ such that the genus of the quotient surface $S/G$ is at most $2$. 
As an application, we show that torsion elements in the mapping class group of a surface of genus $\leq 2$ are conjugacy distinguished.
\bigskip

\noindent {\bf AMS Math Classification:} 57K20, 20F65, 14G32.

\end{abstract}

\maketitle

\section{Introduction}\label{Intro}
\subsection{The congruence subgroup property}
Let $S$ be an oriented connected surface of negative Euler characteristic and finite type and let $G$ 
be a finite group acting faithfully on $S$ by orientation-preserving diffeomorphisms. 
Until Section~\ref{boundary}, we assume that $S$ has empty boundary.
The group $G$ then identifies with a finite subgroup of the mapping class group $\G(S)$.
We write $Z_{\G(S)}(G)$ for its centralizer. The natural action of $\G(S)$ on $\pi_1(S)$ by outer
automorphisms is faithful, and hence so is its restriction to the centralizer.

Taking the profinite completions of $\G(S)$ and $Z_{\G(S)}(G)$ gives the homomorphisms
\begin{equation}\label{proGrepr}
\hG(S)\longrightarrow\Out\bigl(\wh{\pi}_1(S)\bigr)
\qquad\mbox{and}\qquad
\wh{Z}_{\G(S)}(G)\longrightarrow\Out\bigl(\wh{\pi}_1(S)\bigr).
\end{equation}
The \emph{congruence subgroup problem} for $\G(S)$, respectively for $Z_{\G(S)}(G)$,
asks whether the corresponding homomorphism in~\eqref{proGrepr} is injective. 
Equivalently, one asks whether every finite index subgroup of the group in question contains, for some finite index characteristic subgroup $K$ of $\pi_1(S)$, 
the kernel of the induced representation in $\Out(\pi_1(S)/K)$.

The congruence subgroup problem for $\G(S)$, posed by Ivanov~\cite{Ivanov}, is known to have 
a positive answer in genus at most $2$. More precisely, the genus $0$ case is discussed in \cite{Asada, McR}, 
the genus $1$ case is treated in \cite{Asada}, and the genus $2$ case in \cite{hyp, CongTop}. 

The main result of the paper is the following (cf.\ Theorem~\ref{congrel} for a slightly more general result).

\begin{customtheorem}{A}Let $S$ be an oriented connected hyperbolic surface of finite type and $G$ a finite group acting faithfully on $S$ by orientation-preserving diffeomorphisms.
If the congruence subgroup property holds for the mapping class group of the quotient surface $S/G$
with its branch points removed, then it holds for the centralizer $Z_{\G(S)}(G)$.
\end{customtheorem}

\begin{customcorollary}{B}If the quotient $S/G$ has genus at most $2$, then the congruence subgroup property holds for $Z_{\G(S)}(G)$.
\end{customcorollary}

\subsection{Conjugacy separability of torsion elements in mapping class groups}
An element $x$ of a group $G$ is \emph{conjugacy distinguished} if, for every $y\in G$ that is not conjugate to $x$,
there is a finite quotient of $G$ in which their images are not conjugate.
Similarly, a subgroup $H$ of $G$ is \emph{subgroup conjugacy distinguished} if, for every subgroup $K$ of $G$
that is not conjugate to $H$, there is a finite quotient of $G$ in which their images are not conjugate.

Denoting closure in the profinite topology by an overline, we have the following result
(cf.\ Theorem~\ref{torsionrelative} for a slightly more general statement).

\begin{customtheorem}{C}Let $S$ be an oriented connected hyperbolic surface of finite type of genus $\leq 2$. 
\begin{enumerate}
\item The natural monomorphism $\G(S)\hookra\hG(S)$ induces a bijection between conjugacy classes
of solvable finite subgroups. In particular, solvable finite subgroups of $\G(S)$ are subgroup conjugacy distinguished. 
\item Torsion elements of $\G(S)$ are conjugacy distinguished.
\item Let $G$ be a finite subgroup of $\G(S)$. If $S$ has positive genus or $G$ is not isomorphic to the alternating group $A_5$, then
$Z_{\hG(S)}(G)=\ol{Z_{\G(S)}(G)}$.
\end{enumerate}
\end{customtheorem}

\subsection*{Use of generative AI}
The author used OpenAI's Codex to help correct an incomplete proof in the original manuscript.

\section{Preliminary results}
\subsection{Normalizers and centralizers of finite subgroups of the mapping class group of a surface}\label{centrasec}
For a finite set $R$ of points of $S$, let
$\G(S\ssm R;R)$ be the subgroup of $\G(S\ssm R)$ preserving $R$ as a subset of the set of all punctures.
Filling in $R$ defines an epimorphism $\G(S\ssm R;R)\to\G(S)$, since the original punctures of $S$ may be permuted among themselves.
The subgroup preserving any prescribed partition of the punctures is open in the congruence topology.
Indeed, their permutation is detected by the peripheral classes in $H_1(S\ssm R,\Z/3)$:
their sum is zero and this is their only relation, so they are distinct when there are at least two punctures.

We consider connected oriented orbifold surfaces $\mathcal S$ of finite type without boundary, with trivial stabilizers away from the cone points.
Let $S$ be the moduli space of $\mathcal S$. If the cone points have orders $m_1,\ldots,m_r\geq 2$, the Euler characteristic of $\mathcal S$ is
\[
\chi(\mathcal S)=\chi(S)-\sum_{i=1}^r\left(1-\frac{1}{m_i}\right).
\]

Let as above $G$ be a finite group acting faithfully on $S$ by orientation-preserving diffeomorphisms.
The quotient orbifold surface $[S/G]$ has moduli space $S/G$, and the images of its cone points form the branch locus $B$.
Its mapping class group identifies with the subgroup of $\G\bigl((S/G)\ssm B\bigr)$ preserving
the original punctures and the set of cone points of each order.
Let $R$ be the ramification locus of $S\to S/G$. The unramified cover
$S\ssm R\to (S/G)\ssm B$ identifies $\pi_1(S\ssm R)$ with a normal subgroup of
$\pi_1\bigl((S/G)\ssm B\bigr)$ with quotient $G$. The stabilizer $\G_G$ of this covering subgroup in $\G([S/G])$
is the subgroup of mapping classes that lift to $S$. By the Birman--Hilden theorem
(cf.\ \cite[Theorem~2.2 and Proposition~3.1]{MW}), the normalizer of $G$ fits in the exact sequence
\begin{equation}\label{normalizerG}
1\to G\to N_{\G(S)}(G)\to\G_G\to 1.
\end{equation}
The restriction to the centralizer gives the exact sequence
\begin{equation}\label{centralizerG}
1\to Z(G)\to Z_{\G(S)}(G)\to\ker\bigl(\G_G\to\Out(G)\bigr)\to 1.
\end{equation}

\begin{proposition}\label{comparison}
Filling in the ramification points $R$ induces isomorphisms
\[
N_{\G(S\ssm R;R)}(G)\cong N_{\G(S)}(G),\qquad
Z_{\G(S\ssm R;R)}(G)\cong Z_{\G(S)}(G).
\]
\end{proposition}
\begin{proof}
For the unramified cover $S\ssm R\to (S/G)\ssm B$, a lift preserves $R$ if and only if its projection preserves $B$.
The order of inertia over a point of $B$ is determined by the cover. Thus before and after filling,
the groups of lifts are extensions of $\G_G$ by $G$.
Filling induces the identity on the kernel and on the quotient, which proves the first isomorphism.
Since it also identifies the deck transformations, the second follows.
\end{proof}

\subsection{Congruence kernels and finite covers}\label{congtopology}
\begin{lemma}\label{orbifoldnormalizer}
Let $p\colon\mathcal T\to\mathcal S$ be a connected finite regular cover of orbifold surfaces with $\chi(\mathcal S)<0$.
Put $\Pi_{\mathcal S}=\pi_1(\mathcal S)$ and $V=p_*\pi_1(\mathcal T)$. Let $G=\Pi_{\mathcal S}/V$ be the deck group
and $\G_G$ the stabilizer of $V$ in $\G(\mathcal S)$. Descent gives an exact sequence
\[
1\to G\to N_{\G(\mathcal T)}(G)\to\G_G\to1.
\]
\end{lemma}
\begin{proof}
Let $T$ be the moduli space of $\mathcal T$ and $C\subset T$ the set of images of its cone points.
Since $C$ is a union of $G$-orbits, removing $p(C)$ downstairs gives a regular cover
with ordinary hyperbolic surface $T\ssm C$ upstairs.
Apply~\eqref{normalizerG} to the action of $G$ on $T\ssm C$. The group $\G(\mathcal T)$ is the subgroup
of $\G(T\ssm C)$ preserving the subsets of $C$ of each cone order. Thus a class in $\G(\mathcal T)$
normalizing $G$ has a fiber-preserving representative on $T\ssm C$ which extends over $C$ with its given cone orders.
For $x\in C$ the representative preserves both the cone order at $x$ and the local degree of the map of moduli spaces at $x$.
Their product is the cone order at $p(x)$, so the descent extends over $p(C)$ as well.
Since the points of $p(C)$ are cone points of $\mathcal S$, isotopies fixing them as punctures extend over them,
and an isotopy to the identity in $\mathcal S$ restricts to the punctured base. Thus the kernel of descent
is still $G$ by~\eqref{normalizerG}.

By the lifting criterion, a mapping class of $\mathcal S$ lifts precisely when it preserves $V$, and each lift
normalizes $G$. Hence the image of descent is $\G_G$.
\end{proof}

For $H\leq\G_G$, let $E$ be its inverse image under descent, that is, its group of lifts to $\mathcal T$.
Restricting the projection of Lemma~\ref{orbifoldnormalizer} to $E$ gives the short exact sequence
\[
1\to \Pi_{\mathcal S}/V\to E\to H\to1.
\]

More generally, let $V$ be a normal subgroup of finite index in $\Pi_{\mathcal S}$ and suppose that a group $H$
acts on $\Pi_{\mathcal S}$ by outer automorphisms preserving $V$. By a \emph{lift} of $h\in H$ we mean an equivalence class of automorphisms of
$\Pi_{\mathcal S}$ representing the outer action of $h$, where two such automorphisms $\alpha$ and $\beta$ are equivalent
if there exists $v\in V$ such that $\beta(x)=v\alpha(x)v^{-1}$ for every $x\in\Pi_{\mathcal S}$. Pairs consisting of $h$ and a lift
of $h$ form a group $E$ under composition.
Projection onto $H$ gives the same exact sequence, since its kernel consists of inner automorphisms
modulo conjugation by elements of $V$ and $\Pi_{\mathcal S}$ is centerfree.

We now identify these two descriptions of $E$ for $H\leq\G_G$.
For the natural outer action of $H\leq\G_G$, let $f$ represent $h\in H$. The lifts of $f$ to the universal cover
of $\mathcal S$ differ by deck transformations. Since $\Pi_{\mathcal S}$ is centerfree, each automorphism $\alpha$
representing the outer action of $h$ is induced by conjugation by a unique such lift $\widetilde f$.
As $\alpha$ preserves $V$, $\widetilde f$ descends to $\mathcal T$. For $v\in V$, the equivalent automorphism
$x\mapsto v\alpha(x)v^{-1}$ gives the lift $v\widetilde f$ and hence the same map on $\mathcal T$.
Let $f_t$ be an isotopy starting at $f$ and let $\widetilde f_t$ be its lift starting at $\widetilde f$.
For each $\gamma\in\Pi_{\mathcal S}$, the families $\widetilde f_t\gamma$ and $\alpha(\gamma)\widetilde f_t$
lift the same homotopy and agree at $t=0$. By uniqueness of the lifted homotopy they agree for all $t$,
so $\widetilde f_t$ induces $\alpha$ throughout. The resulting homomorphism sends $(h,[\alpha])$
to a mapping class descending to $h$ and sends conjugation by $\gamma\in\Pi_{\mathcal S}$ to the deck transformation
$\gamma V$. It therefore identifies $E$ with the inverse image of $H$ under descent, with the same projection to $H$.
Restriction to $V$ gives the outer action of $E$ on $V$.

Every abstract automorphism of a topologically finitely generated profinite group is continuous by the
strong completeness theorem of Nikolov and Segal~\cite[Theorem~1.1]{NS}. For a hyperbolic surface $S$ of finite type, the finite characteristic quotients of
$\wh{\pi}_1(S)$ define the natural profinite topology on $\Out\bigl(\wh{\pi}_1(S)\bigr)$.
Thus the congruence topology on $\G(S)$ is induced by its natural embedding into this group. We denote the corresponding completion by $\kG(S)$.
The procongruence completion of a subgroup of $\G(S)$ is its closure in $\kG(S)$.

\begin{lemma}\label{transfer}
The projection $E\to H$ induces an isomorphism of congruence kernels
\[
\ker\bigl(\wh E\to\Out(\wh V)\bigr)\xrightarrow{\ \sim\ }
\ker\bigl(\wh H\to\Out(\wh{\Pi}_{\mathcal S})\bigr).
\]
\end{lemma}
\begin{proof}
Choose a torsion-free normal subgroup $L$ of finite index in $\Pi_{\mathcal S}$ contained in $V$.
A deck transformation acting trivially on $H_1(L,\Z/3)$ fixes each puncture, since the peripheral classes are distinct modulo three (their only relation is that their sum is zero).
With exactly two punctures, the two classes are nonzero and opposite, hence distinct modulo three.
It acts trivially on the mod-three homology of the compactification, hence on its Jacobian by Serre's rigidity theorem~\cite{Serre}.
In genus at least two the automorphism group acts faithfully on the Jacobian. In genus one the transformation is a translation fixing a puncture,
and in genus zero it fixes at least three punctures. Thus $\Pi_{\mathcal S}/L$ acts faithfully on $H_1(L,\Z/3)$.
An element of $\wh{\Pi}_{\mathcal S}$ centralizing $\wh L$ therefore lies in $\wh L$ and is trivial,
since $\wh L$ is centerfree (cf.\ \cite[Proposition~3.6]{Magnus}). In particular, $Z_{\wh{\Pi}_{\mathcal S}}(\wh V)=1$.

Let $\Out_{\wh V}(\wh{\Pi}_{\mathcal S})$ denote the group of automorphisms of $\wh{\Pi}_{\mathcal S}$
preserving $\wh V$, modulo conjugation by elements of $\wh V$.
Restriction embeds $\Out_{\wh V}(\wh{\Pi}_{\mathcal S})$ in $\Out(\wh V)$. Indeed, a class in the kernel has a representative
$\alpha$ fixing $\wh V$ pointwise. For every $d\in\wh{\Pi}_{\mathcal S}$, the element $d^{-1}\alpha(d)$
centralizes $\wh V$ and hence is trivial.

The natural map $\Out_{\wh V}(\wh{\Pi}_{\mathcal S})\to\Out(\wh{\Pi}_{\mathcal S})$ has kernel
$\wh{\Pi}_{\mathcal S}/\wh V=\Pi_{\mathcal S}/V$.
On the deck group, the homomorphism $E\to\Out_{\wh V}(\wh{\Pi}_{\mathcal S})$ sends the coset of each element
of $\Pi_{\mathcal S}$ to the class of conjugation by that element. It thus identifies $\Pi_{\mathcal S}/V$
with the kernel just described. Since $\Out_{\wh V}(\wh{\Pi}_{\mathcal S})$ is profinite, this homomorphism factors through $\wh E$.
Hence the finite group $\Pi_{\mathcal S}/V$ embeds in $\wh E$. By right exactness of profinite completion,
we obtain the short exact sequence
\[
1\to \Pi_{\mathcal S}/V\to\wh E\to\wh H\to1.
\]
By injectivity of restriction, an element of $\wh E$ acting trivially in $\Out(\wh V)$
has trivial image in $\Out_{\wh V}(\wh{\Pi}_{\mathcal S})$.
Its projection to $\wh H$ therefore acts trivially in $\Out(\wh{\Pi}_{\mathcal S})$.
The induced map of congruence kernels is injective, since $\ker(\wh E\to\wh H)=\Pi_{\mathcal S}/V$
embeds in $\Out(\wh V)$.

For surjectivity, let $h\in\ker\bigl(\wh H\to\Out(\wh{\Pi}_{\mathcal S})\bigr)$ and choose a lift $e\in\wh E$.
The image of $e$ in $\Out_{\wh V}(\wh{\Pi}_{\mathcal S})$ is mapped to the identity by the natural map 
$\Out_{\wh V}(\wh{\Pi}_{\mathcal S})\to\Out(\wh{\Pi}_{\mathcal S})$.
As the kernel of this map is the deck group $\Pi_{\mathcal S}/V$, there exists $g\in\Pi_{\mathcal S}/V$ such that $g^{-1}e$
has trivial image in $\Out_{\wh V}(\wh{\Pi}_{\mathcal S})$. Thus $g^{-1}e$ still projects to $h$ and acts
trivially in $\Out(\wh V)$.
\end{proof}

\section{Proof of Theorem~A}
\subsection{Abelian covers}
Let $\mathcal S$ be an orbifold surface with $\chi(\mathcal S)<0$. Let $S$ be its moduli space and $B\subset S$ the set of images of its cone points.

\begin{lemma}\label{abelianfill}
If $\mathcal S\simeq[Y/A]$ for a surface $Y$ and a finite abelian group $A$ acting faithfully on $Y$, then the two natural homomorphisms
\[
\hG(\mathcal S)\to\Out\bigl(\wh{\pi}_1(S\ssm B)\bigr),\qquad
\hG(\mathcal S)\to\Out\bigl(\wh{\pi}_1(\mathcal S)\bigr)
\]
have the same kernel.
\end{lemma}
\begin{proof}
Let $H=\ker\bigl(\G_A\to\Aut(A)\bigr)$ as in~\eqref{centralizerG}.
The monodromy epimorphism $\pi_1(\mathcal S)\to A$ and its pullback to $\pi_1(S\ssm B)$
both have stabilizer $H$, so $H$ is open
for both congruence topologies and both kernels are contained in $\wh H$.
By Proposition~\ref{comparison}, removing the ramification points from $Y$ gives the same group of lifts
over $H$, with the same projection to $H$. These lifts have representatives commuting with $A$.

A point of $Y$ is a ramification point of $Y\to S$ precisely when its stabilizer in $A$ is nontrivial.
Every nontrivial stabilizer contains a subgroup of prime order. Thus the ramification locus is the union
of the sets $\Fix_Y(C)$ over all subgroups $C\leq A$ of prime order. Choose an ordering of these subgroups.
At the step corresponding to $C$, remove the points of $\Fix_Y(C)$ that are still present.

Since $A$ is abelian, it preserves each $\Fix_Y(C)$ and hence acts on every intermediate surface.
The lifts over $H$ also preserve these sets, since their representatives can be chosen to commute with $A$.
To see that the lifts define mapping classes at each stage, start with $Y$ punctured at all ramification points.
To recover a given intermediate surface, fill in the punctures corresponding to the ramification points
that still belong to that surface.
The group of lifts preserves this set of punctures, so filling induces a homomorphism to the mapping class group
of the intermediate surface.

At a stage where $C$ has fixed points on the current surface $T$, let $R=\Fix_T(C)$ and let $B_R(T)$
be the fundamental group of the space of unordered configurations of $|R|$ points in $T$, based at $R$.
Filling in $R$ gives the procongruence
Birman exact sequence
\[
1\to\hB_R(T)\to\kG(T\ssm R;R)\to\kG(T)\to1,
\]
where the group in the middle is the closure of the corresponding mapping class group.
The sequence for pure procongruence mapping class groups follows from \cite[Corollary~4.7 and Proposition~4.9]{CongTop}.
The sequence above is then obtained passing to setwise stabilizers.

Since $R$ is the entire ramification locus of $C$ on $T$, Proposition~\ref{comparison} implies that
the only braid in $B_R(T)$ commuting with $C$ in $\G(T\ssm R;R)$ is the identity.
The semidirect product $B_R(T)\rtimes C$ is residually finite, good and of finite virtual
cohomological type. Indeed, it is a finite extension of the pure surface braid group. For the latter,
goodness and finiteness of cohomology with finite coefficients follow by induction from the fibrations
forgetting one point and the goodness of surface groups. The configuration space is a finite-dimensional
$K(\pi,1)$, so the virtual cohomological dimension is finite. For the cohomology comparison, the associated
sequences of fundamental groups remain exact after profinite completion by \cite[Proposition~4.9]{CongTop}.
Residual finiteness also follows from the embedding in $\G(T\ssm R)$. By \cite[Theorem~A(iii)]{BZ},
the centralizer of $C$ in $\hB_R(T)\rtimes C$ is $C$ itself.

If an element of the profinite completion of the group of lifts acts trivially after filling in $R$,
its image before filling lies in $\hB_R(T)$ and commutes with $C$, so it is trivial.
Thus at each stage filling leaves the congruence kernel of the group of lifts unchanged.
Apply Lemma~\ref{transfer} to the cover $Y\to[Y/A]\simeq\mathcal S$ and to its restriction over $S\ssm B$,
whose covering surface is obtained from $Y$ by removing all ramification points.
For both covers the projection from the completed group of lifts to $\wh H$ is the same.
By the lemma, the image of the common congruence kernel under this projection is the kernel of the outer action
of $\wh H$ on $\wh{\pi}_1(\mathcal S)$ in the first case and on $\wh{\pi}_1(S\ssm B)$ in the second.
These kernels therefore coincide. Since both kernels in the statement lie in $\wh H$, this proves the assertion.
\end{proof}

\begin{remark}Note that, in particular, Lemma~\ref{abelianfill} applies to every noncompact $\mathcal S$. In this case
$\pi_1(\mathcal S)$ is the free product of a free group with the cyclic inertia groups.
Projection onto the direct product of the latter is injective on each inertia group and therefore has torsion-free kernel.
\end{remark}

\subsection{The compact case}
We now extend the equality of kernels in Lemma~\ref{abelianfill} to the compact case.

\begin{samepage}
\begin{lemma}\label{compactfill}
Let $\mathcal S$ be a compact orbifold surface with $\chi(\mathcal S)<0$, let $S$ be its moduli space, and
let $B\subset S$ be the set of images of its cone points. Assume that $B$ is nonempty and, when $S$ has genus zero,
contains at least four points. Then the two natural homomorphisms
\[
\hG(\mathcal S)\to\Out\bigl(\wh{\pi}_1(S\ssm B)\bigr),\qquad
\hG(\mathcal S)\to\Out\bigl(\wh{\pi}_1(\mathcal S)\bigr)
\]
have the same kernel.
\end{lemma}
\end{samepage}
\begin{proof}
Choose a cone point $b\in B$ of order $m$, and consider the orbifold surface on $S$ obtained by replacing the order of $b$ by $1$. 
The resulting orbifold surface admits a finite
regular surface cover of degree $d\geq2$. Indeed, the classification of compact orientable orbifold surfaces
in \cite[Theorem~13.3.6]{Thurston} gives a uniformization, since in genus zero at least three cone points remain.
In the spherical case the uniformizing sphere is a regular cover with a finite rotation group as deck group.
In the Euclidean case the translation lattice is normal and of finite index in the uniformizing group,
so it gives a regular torus cover. In the hyperbolic case take a torsion-free normal subgroup of finite index.
Over a cone point of order $n$ every point of the covering surface has local degree $n$, so $n$ divides $d$.
If no cone points remain, $S$ has positive genus, and we take a connected double cover.

Let $T$ be the covering surface, $F$ its deck group and $Q$ the set of $d$ points over $b$.
Giving each point of $Q$ order $m$ defines an orbifold surface $\mathcal T$ with moduli space $T$. The two covers obtained
by removing these points or giving them order $m$ fit in the diagram
\[
\begin{array}{ccccccccc}
1&\to&\pi_1(T\ssm Q)&\to&\pi_1(\mathcal S\ssm\{b\})&\to&F&\to&1\\
 &&\da&&\da&&\Vert\\
1&\to&\pi_1(\mathcal T)&\to&\pi_1(\mathcal S)&\to&F&\to&1 .
\end{array}
\]
The orbifold surface $\mathcal T$ admits a regular abelian surface cover. Indeed, in the usual presentation
send the handle generators to zero and the $d$ cone generators to $e_1,\ldots,e_{d-1},-(e_1+\cdots+e_{d-1})$
in $(\Z/m)^{d-1}$. These images sum to zero and have order $m$, so the kernel is torsion-free.

Let $H\leq\G(\mathcal S)$ be the subgroup of mapping classes lifting to $\mathcal T$.
It is the stabilizer of the kernel of $\pi_1(\mathcal S)\to F$, or equivalently of the kernel
of the composite $\pi_1(S\ssm B)\to\pi_1(\mathcal S)\to F$.
In either profinite fundamental group, choose an open characteristic subgroup $K$ contained in the kernel of the map to $F$.
Preserving this kernel is equivalent to preserving its image in the resulting finite quotient by $K$,
so $H$ is open for both congruence topologies.

The inertia at $b$ maps trivially to $F$, whereas the inertia at every other cone point embeds in $F$.
Thus $H$ fixes $b$. Since $b$ is the only puncture of $\mathcal S\ssm\{b\}$, restriction identifies
the stabilizer of $b$ in $\G(\mathcal S)$ with $\G(\mathcal S\ssm\{b\})$.

The cone points of $\mathcal T$ are precisely the points of $Q$, all of order $m$.
Hence restriction identifies $\G(\mathcal T)$ with $\G(T\ssm Q)$ and preserves the deck subgroup $F$.
Lemma~\ref{orbifoldnormalizer} therefore gives the same normalizer $N_{\G(\mathcal T)}(F)$ for both covers.
Write $p\colon\mathcal T\to\mathcal S$ for the compact cover and $p^\circ$ for its restriction to $T\ssm Q$.
If a fiber-preserving homeomorphism $\widetilde f$ of $\mathcal T$ descends to $f$ on $\mathcal S$,
restricting $p\circ\widetilde f=f\circ p$ gives
\[
p^\circ\circ\bigl(\widetilde f|_{T\ssm Q}\bigr)
=\bigl(f|_{\mathcal S\ssm\{b\}}\bigr)\circ p^\circ.
\]
For the compact cover $p$, which gives the lower row of the diagram, descent sends $[\widetilde f]$ to $[f]\in H$.
For the punctured cover $p^\circ$, which gives the upper row, descent sends $[\widetilde f|_{T\ssm Q}]$
to $[f|_{\mathcal S\ssm\{b\}}]$. Identifying $H$ with its image under restriction, we thus obtain the same
projection $N_{\G(\mathcal T)}(F)\to H$ for both covers.

By Lemma~\ref{abelianfill} applied to $\mathcal T$, the outer actions of $\hG(\mathcal T)$ on
$\wh{\pi}_1(T\ssm Q)$ and $\wh{\pi}_1(\mathcal T)$ have the same kernel.
On each group, an element of $\wh{N}_{\G(\mathcal T)}(F)$ acts through its image in $\hG(\mathcal T)$.
The kernel of either outer action of the completed normalizer
therefore consists of those elements whose image belongs to the common kernel in $\hG(\mathcal T)$.
By Lemma~\ref{transfer} applied to the covers $p$ and $p^\circ$, the projection to $\wh H$ maps this kernel isomorphically
onto the kernel of the outer action of $\wh H$ on $\wh{\pi}_1(\mathcal S)$ for $p$,
and on $\wh{\pi}_1(\mathcal S\ssm\{b\})$ for $p^\circ$. Hence, we have
\[
\ker\!\left(\wh H\to\Out\bigl(\wh{\pi}_1(\mathcal S\ssm\{b\})\bigr)\right)
=
\ker\!\left(\wh H\to\Out\bigl(\wh{\pi}_1(\mathcal S)\bigr)\right).
\]

Removing the remaining cone points from the moduli space of $\mathcal S\ssm\{b\}$ gives $S\ssm B$.
The noncompact case of Lemma~\ref{abelianfill} gives a common kernel for the outer actions of
$\hG(\mathcal S\ssm\{b\})$ on $\wh{\pi}_1(\mathcal S\ssm\{b\})$ and $\wh{\pi}_1(S\ssm B)$.
Since restriction to $S\ssm B$ factors through restriction to $\mathcal S\ssm\{b\}$, the kernels of the
corresponding outer actions of $\wh H$ are both the inverse image of this kernel under
$\wh H\to\hG(\mathcal S\ssm\{b\})$. Since $H$ is open for both congruence topologies,
both kernels on $\hG(\mathcal S)$ lie in $\wh H$. Thus the two homomorphisms in
Lemma~\ref{abelianfill} have the same kernel in the compact case as well.
\end{proof}

\subsection{Conclusion of the proof}
We have $\chi([S/G])=\chi(S)/|G|<0$.
By the hypothesis of Theorem~A, the outer action of $\hG((S/G)\ssm B)$ on $\wh{\pi}_1((S/G)\ssm B)$
is faithful, as is its restriction to the profinite completion of the finite index subgroup $\G([S/G])$.
Lemmas~\ref{abelianfill} and~\ref{compactfill} give faithfulness of the outer action of
$\hG([S/G])$ on $\wh{\pi}_1([S/G])$, except possibly for a compact
sphere with three cone points. When there are no cone points, faithfulness is the hypothesis itself.

For the original cover $S\to[S/G]$, the subgroup $\G_G$ of liftable mapping classes is the stabilizer
of $\pi_1(S)$ in $\G([S/G])$ and hence has finite index. Its group of lifts is $N_{\G(S)}(G)$ by~\eqref{normalizerG}.
By Lemma~\ref{transfer}, its profinite completion $\wh{N}_{\G(S)}(G)$ acts faithfully on
$\wh{\pi}_1(S)$ by outer automorphisms. The congruence subgroup property follows for the normalizer
and hence for its finite index subgroup $Z_{\G(S)}(G)$.
For a compact sphere with three cone points, the orbifold mapping class group is finite, so
\eqref{normalizerG} makes the normalizer finite as well. Its faithful action on $H_1(S,\Z/3)$
gives the same conclusion. This proves Theorem~A.

\section{The congruence subgroup property for the full mapping class group of a surface with boundary}\label{boundary}
Parameterize each component of $\dd S$ by $S^1$ respecting the induced boundary orientation.
The \emph{full mapping class group $\G(S,\dd S)$ of the surface with boundary $(S,\dd S)$} is
the group of isotopy classes of orientation-preserving diffeomorphisms of $S$ whose restrictions to $\dd S$ preserve the given
parameterizations. If $\dd S=\emptyset$, we simply let $\G(S,\dd S)=\G(S)$.

A diffeomorphism preserving the parameterizations restricts to the identity on each boundary component that it sends to itself.
Thus the action on the set of boundary components gives the short exact sequence
\[1\to\Mod(S,\dd S)\to\G(S,\dd S)\to\operatorname{Sym}(\pi_0(\dd S))\to 1,\]
where $\Mod(S,\dd S)$ is the classical relative mapping class group fixing the boundary pointwise.

For $\dd S\neq\emptyset$, the natural representation $\G(S,\dd S)\to\Out(\pi_1(S))$ kills the Dehn twists about
the boundary components. To define the congruence topology in this case, we use the following construction from
\cite[Section~2.7]{Aut}.

Let $\tS$ be the surface without boundary obtained by attaching a $2$-punctured closed disc $D_\delta$ along each
boundary component $\delta$ of $S$. The parameterizations allow us to extend elements of $\G(S,\dd S)$ over the
attached discs, giving an embedding $\G(S,\dd S)\hookra\G(\tS)$. The stabilizer of the collection of attached
discs up to isotopy is the product
\begin{equation}\label{embedding}
\left(\prod_{\delta\in\pi_0(\dd S)}\G(D_\delta,\dd D_\delta)\right)\cdot\G(S,\dd S),
\end{equation}
in which the first factor is normal.

The \emph{congruence topology on $\G(S,\dd S)$} is the topology induced by this embedding from the congruence
topology of $\G(\tS)$. The \emph{procongruence full mapping class group} $\kG(S,\dd S)$ is the associated
completion. The congruence subgroup problem asks whether the natural epimorphism
$\hG(S,\dd S)\to\kG(S,\dd S)$ is an isomorphism.
For a finite subgroup $G$ of $\G(S,\dd S)$, the congruence topology on its centralizer $Z_{\G(S,\dd S)}(G)$ is
again the induced topology. Its completion is the closure of $Z_{\G(S,\dd S)}(G)$ in $\kG(S,\dd S)$.

Replacing the boundary components by punctures gives the interior $\operatorname{int} S$ up to homeomorphism.
We identify the resulting set of punctures with $\pi_0(\dd S)$. Capping gives a short exact sequence
\begin{equation}\label{relexseq1}
1\to\Z[\pi_0(\dd S)]\to\G(S,\dd S)\to\G(\operatorname{int} S;\pi_0(\dd S))\to 1,
\end{equation}
where each boundary component represents the Dehn twist about that component and the group on the right
preserves the set of boundary punctures. By \cite[Section~2.7]{Aut}, procongruence completion gives the
short exact sequence
\begin{equation}\label{relexseq2}
1\to\ZZ[\pi_0(\dd S)]\to\kG(S,\dd S)\to\overline{\G(\operatorname{int} S;\pi_0(\dd S))}\to 1,
\end{equation}
where the closure on the right is taken in $\kG(\operatorname{int} S)$.

\begin{lemma}\label{boundarylifts}
Let $G\leq\G(S,\dd S)$ be finite and identify it with its image under capping. Capping gives the short exact sequence
\begin{equation}\label{relnorm}
1\to\Z[\pi_0(\dd S)]^G\to N_{\G(S,\dd S)}(G)\to N_{\G(\operatorname{int} S;\pi_0(\dd S))}(G)\to 1 .
\end{equation}
If an element of $\G(\operatorname{int} S;\pi_0(\dd S))$ centralizes $G$, its normalizing lifts centralize $G$.
The kernel of capping on the centralizer of $G$ in $\G(S,\dd S)$ is also $\Z[\pi_0(\dd S)]^G$.
\end{lemma}

\begin{proof}
Since the kernel of~\eqref{relexseq1} is torsion-free, we identify $G$ with its image under capping.
Taking fixed points for the conjugation action of $G$ in~\eqref{relexseq1} gives the exact sequence of pointed sets
\[
1\to\Z[\pi_0(\dd S)]^G\to Z_{\G(S,\dd S)}(G)\to Z_{\G(\operatorname{int} S;\pi_0(\dd S))}(G)
\to H^1(G,\Z[\pi_0(\dd S)]).
\]
This action permutes the twists about the boundary components. Choose one component $\delta$ in each $G$-orbit and let $G_\delta$ be its stabilizer.
The twists in the orbit of $\delta$ span the permutation module $\Z[G/G_\delta]$, and the group of boundary
twists is the direct sum of these modules. Sending each coset to its characteristic function identifies $\Z[G/G_\delta]$ with
the finitely supported integer-valued functions on the orbit. Since the orbit is finite, these are all the functions
$G/G_\delta\to\Z$, which form the module coinduced from the trivial $G_\delta$-module $\Z$. Shapiro's lemma therefore gives
\[
H^1(G,\Z[\pi_0(\dd S)])\cong\bigoplus_\delta H^1(G_\delta,\Z)
=\bigoplus_\delta\Hom(G_\delta,\Z)=0,
\]
since every $G_\delta$ is finite. Thus capping is surjective on centralizers, with kernel $\Z[\pi_0(\dd S)]^G$.

Let $\widetilde f$ lift a mapping class $f\in\G(\operatorname{int} S;\pi_0(\dd S))$ normalizing the image of $G$. 
In the preimage $\Z[\pi_0(\dd S)]\rtimes G$, both $G$ and $\widetilde fG\widetilde f^{-1}$ are complements to the
group of boundary twists. A complement is the image of a unique homomorphic section $g\mapsto c(g)g$ of capping,
with $c(g)\in\Z[\pi_0(\dd S)]$.
Such sections correspond to $1$-cocycles, since $g\mapsto c(g)g$ is a homomorphism precisely when
$c(gh)=c(g)\,g c(h)g^{-1}$. Replacing a complement $C$ by $b^{-1}Cb$, where $b$ is a product of boundary
twists, replaces $c(g)$ by $c(g)(b^{-1}gbg^{-1})$. Thus conjugation multiplies $c$ by the coboundary
$d^0b$, and the conjugacy classes of complements under the boundary twists correspond to $H^1(G,\Z[\pi_0(\dd S)])$. 
Hence some product $a$ of boundary twists conjugates $\widetilde fG\widetilde f^{-1}$
to $G$, so $a\widetilde f$ is a normalizing lift.
If the capped class centralizes $G$, every normalizing lift centralizes $G$ by injectivity of capping restricted to $G$.
If $x$ normalizes $G$ and caps to the identity, then for every $g\in G$ the elements $xgx^{-1}$ and $g$
belong to $G$ and have the same capped image. Hence $xgx^{-1}=g$, so $x$ lies in the kernel on the centralizer.
The reverse inclusion follows since the centralizer is contained in the normalizer. This common kernel
$\Z[\pi_0(\dd S)]^G$ and the surjectivity on normalizers proved above give the short exact sequence~\eqref{relnorm}.
\end{proof}

We then have the following generalization of Theorem~A and Corollary~B.

\begin{theorem}\label{congrel}Let $S$ be an oriented connected hyperbolic surface of finite type and $G$ a finite subgroup of
$\G(S,\dd S)$. If the mapping class group of the quotient $(\operatorname{int} S)/G$ with the branch points removed
has the congruence subgroup property, then so does $Z_{\G(S,\dd S)}(G)$.
In particular, if the quotient has genus at most $2$, the natural map identifies
$\wh{Z}_{\G(S,\dd S)}(G)$ with the closure of $Z_{\G(S,\dd S)}(G)$ in $\kG(S,\dd S)$.
\end{theorem}

\begin{proof}
Identify $G$ with its image under capping. By Lemma~\ref{boundarylifts}, capping gives the short exact sequence
\[
1\to\Z[\pi_0(\dd S)]^G\to Z_{\G(S,\dd S)}(G)\to Z_{\G(\operatorname{int} S;\pi_0(\dd S))}(G)\to 1 .
\]
The products of the boundary twists over the individual $G$-orbits form a basis of the kernel. By~\eqref{relexseq2},
its closure in $\kG(S,\dd S)$ is $\ZZ[\pi_0(\dd S)]^G$, its full profinite completion. In particular, the
displayed sequence remains exact after profinite completion. The group on the right is the subgroup of
$Z_{\G(\operatorname{int} S)}(G)$ preserving the boundary punctures. It is open in the congruence topology
and has the congruence subgroup property by Theorem~A. Thus every element of the congruence kernel of
$Z_{\G(S,\dd S)}(G)$ lies in $\ZZ[\pi_0(\dd S)]^G$, which embeds in $\kG(S,\dd S)$ by~\eqref{relexseq2}.
\end{proof}

\begin{lemma}[Descent for surfaces with boundary]\label{boundarydescent}
Each finite subgroup $G\leq\G(S,\dd S)$ can be realized by an action on $S$ preserving the boundary parameterizations.
Let $B$ be the branch locus and $\G_G^\partial$ the subgroup of $\G((S/G)\ssm B,\dd(S/G))$ preserving $B$
and the kernel of the monodromy $\pi_1((S/G)\ssm B)\to G$. Descent gives a short exact sequence
\begin{equation}\label{boundarydescentseq}
1\to G\to N_{\G(S,\dd S)}(G)\to\G_G^\partial\to 1.
\end{equation}
In particular, $N_{\G(S,\dd S)}(G)$ is commensurable with $\G((S/G)\ssm B,\dd(S/G))$.
\end{lemma}

\begin{proof}
First observe that $G$ acts freely on the set of boundary components. Indeed, if $g\in G$ preserves a component
$\delta$, cap all the other components by punctured discs. The image of $g$ fixes $\delta$ pointwise and has
finite order, so it is trivial by \cite[Corollary~7.3]{FM}. Thus $g$ also has trivial image after full capping.
Since full capping is injective on $G$ by~\eqref{relexseq1}, we have $g=1$.

By Nielsen realization, the capped group acts on $\operatorname{int} S$. Fill the boundary punctures and remove
the interiors of disjoint discs about them chosen equivariantly. Since the new boundary circles form free
$G$-orbits, we may parameterize one circle in each orbit and transport its parameterization by $G$.
Identifying these circles with the parameterized boundary components of $S$ gives a finite group of lifts of
the capped action. This group and the given $G$ are complements to the group of boundary twists in the preimage
of the capped group. By the proof of Lemma~\ref{boundarylifts}, a product of boundary twists conjugates one
complement to the other. Conjugating the action accordingly realizes the given $G$ on $S$.

Each boundary circle maps with degree one onto its image in $S/G$, to which we give the induced parameterization.
By the covering criterion, each class in $\G_G^\partial$ lifts off the branch locus. Its lifts extend over the
ramification points, preserve the boundary parameterizations and differ by elements of $G$. Lifting isotopies
therefore gives a homomorphism $\phi\co\G_G^\partial\to N_{\G(S,\dd S)}(G)/G$.
To see that this is an isomorphism, we first consider the homomorphism induced by $\phi$ on the capped surfaces and
then compare the kernels of the capping homomorphisms. 

On the capped surfaces, \eqref{normalizerG} identifies $N_{\G(\operatorname{int} S)}(G)/G$ with the group $\G_G$ of mapping classes
of the quotient orbifold surface $[(\operatorname{int} S)/G]$ that lift to $\operatorname{int} S$. A lift preserves
the boundary punctures if and only if its descent preserves their images, since they form a union of whole
fibers. A lift by the capping sequence~\eqref{relexseq1} for the quotient surface $(S/G)\ssm B$ of an element of $\G_G$ which
preserves the images of the boundary punctures preserves both $B$ and the kernel of $\pi_1((S/G)\ssm B)\to G$.
It therefore belongs to $\G_G^\partial$.
Thus capping the boundary of $S/G$ maps $\G_G^\partial$ onto the subgroup of $\G_G$ preserving the images
of the boundary punctures. By~\eqref{normalizerG}, the lifts of these classes, taken modulo $G$, form the group
$N_{\G(\operatorname{int} S;\pi_0(\dd S))}(G)/G$. By~\eqref{relnorm}, capping the boundary of $S$ maps
$N_{\G(S,\dd S)}(G)/G$ onto this same group.

It remains to compare the kernels of capping. A twist about a boundary circle of $S/G$ lifts to the product of the twists
about the boundary circles of $S$ in its preimage, since each maps with degree one. The boundary twists of $S/G$
form a basis of the kernel of capping on $\G_G^\partial$ by~\eqref{relexseq1}, and their lifts form a basis
of the kernel of capping on $N_{\G(S,\dd S)}(G)/G$ by~\eqref{relnorm}. Therefore $\phi$ restricts to an isomorphism 
on the kernels of capping as well. This proves~\eqref{boundarydescentseq}.

Preserving $B$ is a finite-index condition as is preserving the kernel of the epimorphism $\pi_1((S/G)\ssm B)\to G$. 
Thus $\G_G^\partial$ has finite index in $\G((S/G)\ssm B,\dd(S/G))$. Intersect $N_{\G(S,\dd S)}(G)$ with a torsion-free subgroup of finite index in $\G(S,\dd S)$.
The intersection has finite index in the normalizer and meets the finite group $G$ trivially.
 The short exact sequence~\eqref{boundarydescentseq} then identifies this intersection with a finite-index subgroup of
$\G_G^\partial$, hence of $\G((S/G)\ssm B,\dd(S/G))$. This proves commensurability.
\end{proof}

\newpage
\section{Conjugacy separability}
We will prove the following more general version of Theorem~C.

\begin{theorem}\label{torsionrelative}For $S$ an oriented connected hyperbolic surface of finite type of genus $\leq 2$, we have:
\begin{enumerate}
\item The natural monomorphism $\G(S,\dd S)\hookra\hG(S,\dd S)$ induces a bijection between conjugacy classes
of solvable finite subgroups. In particular, solvable finite subgroups of $\G(S,\dd S)$ are subgroup conjugacy distinguished.
\item Torsion elements of $\G(S,\dd S)$ are conjugacy distinguished.
\item Let $G$ be a finite subgroup of $\G(S,\dd S)$. If $S$ has positive genus or $G\not\cong A_5$, then
\[Z_{\hG(S,\dd S)}(G)=\ol{Z_{\G(S,\dd S)}(G)}.\]
\end{enumerate}
\end{theorem}

\begin{proof}
We apply the criterion of \cite[Corollary~D]{BZ}. The ambient group $\G(S,\dd S)$ must be residually finite,
good and of finite virtual cohomological type. For each finite solvable subgroup $G$, we must verify two
further properties of its normalizer. First, the normalizer must itself be good and of finite virtual
cohomological type. Second, the profinite topology of $\G(S,\dd S)$ must induce the full profinite topology
on the normalizer. We use Lemma~\ref{boundarydescent} for the first property and Theorem~\ref{congrel} for the second.

The group $\G(S,\dd S)$ is residually finite and of finite virtual cohomological type. Since $S$ has genus
at most $2$, its finite-index subgroup fixing the boundary and the punctures pointwise is good by
\cite[Proposition~6.1]{BZ}. Hence $\G(S,\dd S)$ is good as well.

Let $G\leq\G(S,\dd S)$ be finite and solvable. By Lemma~\ref{boundarydescent}, its normalizer is
commensurable with $\G((S/G)\ssm B,\dd(S/G))$, where $B$ is the branch locus.
By Riemann--Hurwitz, $S/G$ also has genus at most $2$, so the preceding paragraph applies to this
mapping class group. Goodness and finite virtual cohomological type are invariant under commensurability.
This proves the first property.

For the second property, Theorem~\ref{congrel} gives the full profinite topology on $Z_{\G(S,\dd S)}(G)$.
By residual finiteness, choose a finite quotient of $\G(S,\dd S)$ which is injective on $G$. An element of
the normalizer in the kernel centralizes $G$, since conjugation by it fixes the image of each $g\in G$ in
this quotient. Thus the centralizer is open in the normalizer for the induced topology. Since this topology
is the full profinite topology on the centralizer, it is so on the normalizer as well.

Assertions (i) and (ii) now follow from \cite[Corollary~D]{BZ}. For (iii), the same corollary gives
\[
N_{\hG(S,\dd S)}(G)=\ol{N_{\G(S,\dd S)}(G)}\cong\wh{N}_{\G(S,\dd S)}(G).
\]
The conjugation action of the profinite normalizer on $G$ has finite image. Its kernel is open, so the
intersection of this kernel with the discrete normalizer is dense in it. This gives
$Z_{\hG(S,\dd S)}(G)=\ol{Z_{\G(S,\dd S)}(G)}$.

To conclude, we recall why every finite subgroup occurring in assertion~(iii) is solvable. By Nielsen realization
on the capped surface and extension across its punctures, every finite subgroup $G$ of $\G(S,\dd S)$ acts
faithfully on a compact Riemann surface $X$ of the same genus as $S$.
In genus $0$, the finite subgroups of $\Aut(X)=\PGL_2(\C)$ are cyclic groups, dihedral groups, $A_4$, $S_4$
and $A_5$. All except the last are solvable. In genus $1$, the translations in $G$ form an abelian normal subgroup $A$
and the quotient $G/A$ is cyclic of order $1$, $2$, $3$, $4$ or $6$, so $G$ is solvable.

Finally, suppose that $S$ has genus $2$. The hyperelliptic involution $\iota$ is central in $\Aut(X)$, and the quotient
of $G$ by $G\cap\langle\iota\rangle$ is a finite subgroup of $\PGL_2(\C)$ preserving the six branch points of
$X\to X/\langle\iota\rangle$. By the classification of finite subgroups of $\PGL_2(\C)$ recalled above, it suffices to rule out
$G/(G\cap\langle\iota\rangle)\cong A_5$. The stabilizer of a point of $\P^1$ in a finite group of automorphisms is cyclic and
every cyclic subgroup of $A_5$ has order at most $5$, so every $A_5$-orbit on $\P^1$ has at least $12$ points.
The quotient group therefore cannot be isomorphic to $A_5$ and is solvable. Since $G\cap\langle\iota\rangle$ has order
at most $2$, $G$ is solvable in genus $2$ as well, proving assertion~(iii).
\end{proof}

\bigskip

\subsection*{Conflict of interest declaration:} The author declares that he has no conflict of interest.

\end{document}